\documentclass[preprint,review, 10pt,3p,times]{article}

\usepackage{amssymb,amsthm,amsmath,amsfonts,graphicx}
\usepackage[english]{babel}
\usepackage{latexsym}
\usepackage{url}
\usepackage[T1]{fontenc}
\usepackage{color}
\usepackage{enumerate}
\newcommand{\bc}{\begin{center}}
\newcommand{\ec}{\end{center}}
\newcommand{\cc}{\textrm{cc}}
\newcommand{\cy}{\textrm{cy}}
\newcommand{\ex}{\textrm{ex}}

\newtheorem{theorem}{Theorem}
\newtheorem{proposition}{Proposition}
\newtheorem{corollary}{Corollary}
\newtheorem{lemma}{Lemma}

\newtheorem{defi}{Definition}

\begin{document}

%\begin{frontmatter}

\bc
{\Large
On global location-domination in bipartite graphs}
\ec

\bc
{\large
Carmen Hernando, Merce Mora, Ignacio M. Pelayo}
\ec

\bc
{\normalsize
Universitat Politècnica de Catalunya, Barcelona, Spain}
\ec

%\author[upc]{C.~Hernando}
%\ead{carmen.hernando@upc.edu}

%\author[upc]{M.~Mora}
%\ead{merce.mora@upc.edu}

%\author[upc]{I. M.~Pelayo\corref{cor1}}
%\ead{ignacio.m.pelayo@upc.edu}


%\address[upc]{Universitat Politècnica de Catalunya, Barcelona, Spain}

\author{}

%\address{}

%%%%%%%%%%%%%%%%%%%%%%%%%%%%%%%%%%%%%%%%%%%%%%%%%%%%%%%%%%%%%%%%%%%%
%%%%%%%%%%%%%%%%%%%%%%%%%%%%%%%%%%%%%%%%%%%%%%%%%%%%%%%%%%%%%%%%%%%%
\begin{abstract}
\small
A dominating set $S$  of a graph $G$ is called \emph{locating-dominating}, \emph{LD-set} for short, if every vertex $v$ not in  $S$ is uniquely determined by the set of neighbors of $v$ belonging to $S$.
Locating-dominating sets of minimum cardinality are called $LD$-codes and the cardinality of an LD-code is the \emph{location-domination number} $\lambda(G)$.
An LD-set $S$ of a graph $G$ is  \emph{global} if it is an LD-set of both $G$ and its complement $\overline{G}$.
The   \emph{global location-domination number} $\lambda_g(G)$ is the minimum cardinality of a global LD-set of $G$.

For any LD-set $S$ of a given graph $G$, the so-called \emph{S-associated graph} $G^S$ is introduced. This edge-labeled bipartite graph turns out to be very helpful to approach the study of LD-sets in graphs, particularly when $G$ is bipartite.

This paper is mainly devoted  to the study of relationships between  global LD-sets, LD-codes and the location-domination number in a graph $G$ and its complement $\overline{G}$, when $G$ is bipartite.
%\end{abstract}

\vspace{.3cm}
%\begin{keyword}
\noindent {\bf Keywords}: \small Domination, Global domination, Locating domination, Complement graph, Bipartite graph.
%\end{keyword}
\end{abstract}

%\end{frontmatter}

%%%%%%%%%%%%%%%%%%%%%%%%%%%%%%%%%%%%%%%%%%%%%%%%%%%%%
%%%%%%%%%%%%%%%%%%%%%%%%%%%%%%%%%%%%%%%%%%%%%%%%%%%%%
%%%%%%%%%%%%%%%%%%%%%%%%%%%%%%%%%%%%%%%%%%%%%%%%%%%%%
%%%%%%%%%%%%%%%%%%%%%%%%%%%%%%%%%%%%%%%%%%%%%%%%%%%%%
%%%%%%%%%%%%%%%%%%%%%%%%%%%%%%%%%%%%%%%%%%%%%%%%%%%%%
%%%%%%%%%%%%%%%%%%%%%%%%%%%%%%%%%%%%%%%%%%%%%%%%%%%%%
\section{Introduction}

Let  $G=(V,E)$ be a simple, finite graph.
The \emph{open neighborhood} of a vertex $v\in V$ is $N_G(v)=\{u\in V : uv\in E\}$.
%and the \emph{close neighborhood} is $N_G[v]=\{u\in V : uv\in E\}\cup \{ v \}$.
The \emph{complement} of a graph $G$,  denoted by $\overline{G}$, is the  graph on the same vertices such that two vertices are adjacent in $\overline{G}$ if and only if they are not adjacent in $G$.
The distance between vertices $v,w\in V$ is denoted by $d_G(v,w)$.
We write $N(u)$ or $d(v,w)$ if the graph G is clear from the context.
Given any pair of sets  $A$ and $B$,  $A\bigtriangleup B $ denotes its symmetric difference, that is, $(A\setminus B) \cup (B \setminus A)$.
For further notation and terminology , we refer the reader to \cite{chlezh11}.

A set $D\subseteq V$ is a \emph{dominating set} if for every vertex $v\in V\setminus D$, $N(v)\cap D\neq\emptyset$.
The \emph{domination number} of $G$, denoted by $\gamma(G)$, is the minimum cardinality of a dominating set of $G$ \cite{hahesl}.
A dominating set is \emph{global} if it is a dominating set of both $G$ and its complement graph, $\overline{G}$.
The minimum cardinality of a global dominating set of $G$, denoted by $\gamma_g (G)$, is the \emph{global domination number} of $G$ \cite{brca98,brdu90,sam89}.
If $D$ is a subset of $V$ and $v\in V\setminus D$, we say that  $v$ \emph{dominates}  $D$ if $D\subseteq N(v)$.

A dominating set $S\subseteq V$ is  a \emph{locating-dominating set}, \emph{LD-set} for short, if for every two different vertices $u,v\in V\setminus S$, $N(u)\cap S\neq N(v)\cap S$.
The \emph{location-domination number} of $G$,   denoted by $\lambda (G)$, is the  minimum cardinality of a locating-dominating set.
A locating-dominating set of cardinality $\lambda(G)$ is called an \emph{LD-code}~\cite{rasl84,slater88}.
Certainly,  every LD-set of a non-connected graph $G$ is  the  union of LD-sets of its connected components and the location-domination number is the sum of the location-domination number of its connected components.
LD-codes and the location-domination parameter have been intensively studied during the last decade;
see \cite{bchl,bcmms07,cahemopepu12,clm11,hola06,ours3}.
A complete and regularly updated list of papers on locating-dominating codes is to be found in \cite{lobstein}.

%The goal of...

The remaining part of this paper is organized as follows.
In Section 2,  we deal with the problem of approaching the relationship between $\lambda (G)$ and $\lambda (\overline{G})$, for any arbitrary graph $G$.
In Section 3, we introduce the so-called  \emph{LD-set-associated graph} $G^S$, which is an edge-labeled bipartite graph constructed from an arbitrary LD-set $S$  of a given graph $G$,  and show some basic properties of this graph.
Finally, Section 4 is  concerned with the study of relationships between the location-domination number $\lambda (G)$ of a bipartite graph $G$ and the location-domination number $\lambda (\overline{G})$ of its complement $\overline{G}$.

%%%%%%%%%%%%%%%%%%%%%%%%%%%%%%%%%%%%%%%%%%%%%%%%%%%%%
%%%%%%%%%%%%%%%%%%%%%%%%%%%%%%%%%%%%%%%%%%%%%%%%%%%%%
%%%%%%%%%%%%%%%%%%%%%%%%%%%%%%%%%%%%%%%%%%%%%%%%%%%%%
%%%%%%%%%%%%%%%%%%%%%%%%%%%%%%%%%%%%%%%%%%%%%%%%%%%%%
%%%%%%%%%%%%%%%%%%%%%%%%%%%%%%%%%%%%%%%%%%%%%%%%%%%%%
%%%%%%%%%%%%%%%%%%%%%%%%%%%%%%%%%%%%%%%%%%%%%%%%%%%%%
\section{General case}

This section is devoted to approach the relationship between $\lambda (G)$ and $\lambda (\overline{G})$, for any arbitrary graph $G$.
Some of the results we present were previously shown  in \cite{ours3,oursglobal} and we include them for the sake of completeness.

%%%%%%%%%%%%%%%%%%%%%%%%%%%%%%% els conjunts $N_G(u) \cap S$ són diferents si i nomes si $N_{G compl} \cap S$ son diferents

Notice that $N_{\overline{G}}(x)\cap S = S\setminus N_{G}(x)$ for any set $S\subseteq V$ and any vertex $x\in V\setminus S$.
A straightforward consequence of this fact are the following results.

\begin{proposition} [\cite{oursglobal}]\label{pro.domi}
If $S\subseteq V$ is an LD-set of a graph $G=(V,E)$, then $S$ is an LD-set of $\overline{G}$ if and only if $S$ is a dominating set of $\overline{G}$.
\end{proposition}

\begin{proposition} [\cite{ours3}]\label{pro.vertexdom}
Let  $S\subseteq V$ be an LD-set of a graph $G=(V,E)$. Then, the following holds.
\begin{itemize}
\item[(a)]  There is at most one vertex $u\in V\setminus S$ dominating $S$, and in the case it exists, $S\cup \{ u \}$ is an LD-set of $\overline{G}$.
\item[(b)] $S$ is an LD-set of $\overline{G}$ if and only if there is no vertex in $V\setminus S$ dominating $S$ in $G$.
\end{itemize} 
\end{proposition}
%%%%%%%%%%%%%%%%%%%%%%%%%%%%%%%%%%%%%%%%%%%%%%%%%%%%%

The following theorem is a consequence of the preceding propositions.

%%%%%%%%%%%%%%%%%%%%%%%%%%%%%%%%%%%%%%%%%%%%%%%%%%%%%
\begin{theorem} [\cite{ours3}]\label{cor.difuno}
For every graph $G$, $|\lambda (G) -\lambda (\overline{G})|\le 1$.
\end{theorem}
%%%%%%%%%%%%%%%%%%%%%%%%%%%%%%%%%%%%%%%%%%%%%%%%%%%%%

According to the preceding inequality, for every graph $G$, $\lambda (\overline{G})\in\{\lambda (G)-1,\lambda (G),\lambda (G)+1\}$, all cases being  feasible for some connected graph $G$.
See  Table \ref{tab.valors} for some basic examples covering all possible cases.

We intend to obtain  either necessary or sufficient conditions for a graph $G$ to satisfy $\lambda (\overline{G})>\lambda (G)$, i.e., $\lambda (\overline{G})= \lambda (G) +1$.
This problem was approached and completely solved in \cite{oursglobal} for the family of block-cactus.
In this work, we carry out a similar study for bipartite graphs.
After noticing that solving the equality $\lambda (\overline{G})= \lambda (G) +1$ is closely related to analyzing the existence or not  of sets that are simultaneously locating-dominating sets in both $G$ and its complement  $\overline{G}$, the following definitions were introduced in \cite{oursglobal}.

\begin{defi} [\cite{oursglobal}] \rm
A set $S$ of vertices of a graph $G$  is a \emph{global LD-set} if $S$ is an LD-set of both $G$ and its complement $\overline{G}$.
The \emph{global location-domination number} of a graph $G$, denoted by $\lambda_g (G)$, is defined as the minimum cardinality of a global LD-set of $G$.
\end{defi}

According to Proposition \ref{pro.vertexdom}, an LD-set $S$ of a graph $G$ is non-global if and only if there exists a (unique)  vertex $u\in V(G)\setminus S$ which dominates $S$, i.e., such that $S\subseteq N(u)$.
Notice  that, for every graph $G$, $\lambda_g (\overline{G})=\lambda_g (G)$, 
since for  every set of vertices $S\subset V(G)=V(\overline{G})$, $S$ is a global LD-set of  $G$ if and only if it  is a global LD-set of   $\overline{G}$.
Observe also that an LD-code $S$ of $G$ is a global LD-set  if and only if it is both an LD-code of $G$ and an LD-set of $\overline{G}$.

%%%%%%%%%%%%%%%%%%%%%%%%%%%%%%%%%%%%%%%%%%%%%%%%%%%%%
\begin{theorem} [\cite{oursglobal}]\label{teoremon}
For any graph $G=(V,E)$,
$ \max \{ \lambda (G), \lambda (\overline{G}) \} \le \lambda_g (G)  \le \min \{ \lambda ({G})+1, \lambda (\overline{G})+1\}.$
Moreover,
\begin{itemize}

\item[(a)] If $\lambda (G)\not= \lambda (\overline{G})$, then $\lambda_g(G)=\max \{ \lambda (G), \lambda (\overline{G})\}$.

\item[(b)] If $\lambda (G)=\lambda (\overline{G})$, then $\lambda_g(G)\in \{ \lambda (G), \lambda (G) +1 \}$, and both possibilities are feasible.

\item[(c)] $\lambda_g (G) = \lambda ({G})+1$ if and only if every LD-code of $G$ is non-global.

\end{itemize}
\end{theorem}
%%%%%%%%%%%%%%%%%%%%%%%%%%%%%%%%%%%%%%%%%%%%%%%%%%%%%

\begin{corollary}\label{cori}
If $G$ is a graph with a global LD-code, then $\lambda (\overline{G})\le \lambda (G)$.
\end{corollary}

In Table \ref{tab.valors},  the location-domination number of some families of graphs is displayed, along with the location-domination number of its complement graphs and the global location-domination number.
Concretely, we consider the path $P_n$ of order  $n\ge7$; the cycle $C_n$ of order  $n\ge7$;  the wheel  $W_n$  of order  $n\ge8$, obtained by joining a new vertex to all vertices of a cycle of order $n-1$; the complete graph  $K_n$ of order  $n\ge2$; the complete bipartite graph $K_{r,n-r}$ of order  $n\ge4$, with $2\le r\le n-r$ and stable sets of order $r$ and $n-r$, respectively; the star $K_{1,n-1}$ of order  $n\ge4$, obtained by joining a new vertex to $n-1$ isolated vertices; and finally, the bi-star $K_2(r,s)$ of order $n\ge6$ with  $3\le r\le s=n-r$,  obtained by joining the central vertices of two stars $K_{1,r-1}$ and $K_{1,s-1}$ respectively.

\begin{proposition} [\cite{oursglobal}]\label{donosti}
 Let   $G$ be a graph of order $n$.
 If $G \in \{P_n,C_n,W_n,K_n,K_{1,n-1},K_{r,n-r},K_2(r,s)\}$,
then the values of $\lambda (G)$, $\lambda (\overline{G})$ and  $\lambda_g (G)$ are known and they are displayed in Table \ref{tab.valors}.
\end{proposition}

{\tiny
\begin{table}[hbt]
\begin{center}
  \begin{tabular}{c|ccccccc}
          { $G$} &  { $P_n$} &  { $C_n$} & { $W_n$}  & { $K_n$} & { $K_{1,n-1}$} &  { $K_{r,n-r}$}  & {$K_2${$(r,s)$}} \\
                  \hline
				{ $n$}	&   { $n\ge 7$} & { $n\ge 7$} & { $n\ge 8$} & { $n\ge 2$} & { $n\ge 4$} &  { $2\le r\le n-r$} & { $3\le r\le s$} \\
                  \hline			
         { $\lambda (G)$} & { $\lceil \frac {2n}5 \rceil$}   & { $\lceil \frac {2n}5 \rceil$} & { $\lceil \frac {2n-2}5 \rceil$}  & { $n-1$} &  { $n-1$} & { $n-2$} & { $n-2$} \\
				%&&&&&&&\\
         %
      { $\lambda (\overline{G})$} & { $\lceil \frac {2n-2}5 \rceil$} & { $\lceil \frac {2n-2}5 \rceil$} & { $ \lceil \frac {2n+1}5 \rceil$}  & { $n$} &  { $n-1$} & { $n-2$} & { $n-3$} \\
					
				%&&&&&&&\\
         %
         { $\lambda_g (G)$ }
         & { $\lceil \frac {2n}5 \rceil$}  & { $\lceil \frac {2n}5 \rceil$}  &  { $\lceil \frac {2n+1}5 \rceil$} & { $n$}    &  { $n-1$} & { $n-2$}  & { $n-2$} \\
         \hline
\end{tabular}
\end{center}
\caption{The values of  $\lambda (G)$, $\lambda (\overline{G})$ and $\lambda_g (G)$ for some families of graphs.}
\label{tab.valors}
\end{table}}

%\newpage
%%%%%%%%%%%%%%%%%%%%%%%%%%%%%%%%%%%%%%%%%%%%%%%%%%%%%
%%%%%%%%%%%%%%%%%%%%%%%%%%%%%%%%%%%%%%%%%%%%%%%%%%%%%
%%%%%%%%%%%%%%%%%%%%%%%%%%%%%%%%%%%%%%%%%%%%%%%%%%%%%
%%%%%%%%%%%%%%%%%%%%%%%%%%%%%%%%%%%%%%%%%%%%%%%%%%%%%
%%%%%%%%%%%%%%%%%%%%%%%%%%%%%%%%%%%%%%%%%%%%%%%%%%%%%
%%%%%%%%%%%%%%%%%%%%%%%%%%%%%%%%%%%%%%%%%%%%%%%%%%%%%
\section{The LD-set-associated graph}

Let $S $ be an LD-set of a graph $G$.
We introduce in this section a labeled graph associated to $S$ and study some general properties.
This graph will allow us to derive  some properties related to LD-sets and the location-domination number of $G$.

\begin{defi}\label{def.Gomega} \rm
Let $S$ be  an LD-set with exactly $k$ vertices of a connected graph $G=(V,E)$ of order $n$.
Consider $z\notin V(G)$ and define $N_G(z)=\emptyset$.
The so-called \emph{$S$-associated graph}, denoted by $G^{S}$, is the edge-labeled graph defined as follows.
\begin{itemize}
\item[(1)] $V(G^{S})=(V\setminus S)\cup \{z\}$;
\item[(2)]  For every pair of vertices $x,y\in V(G^{S})$, $xy\in E(G^{ S})$ if and only if $|(N_G(x)\cap S)\bigtriangleup (N_G(y)\cap S)|=1$;
\item[(3)]  The label $\ell (xy)$ of edge $xy\in E(G^{ S})$ is the only element of $(N_G(x)\cap S)\bigtriangleup (N_G(y)\cap S)\in S$.
\end{itemize}
\end{defi}

\begin{figure}[!hbt]
\begin{center}
\includegraphics[height=5cm]{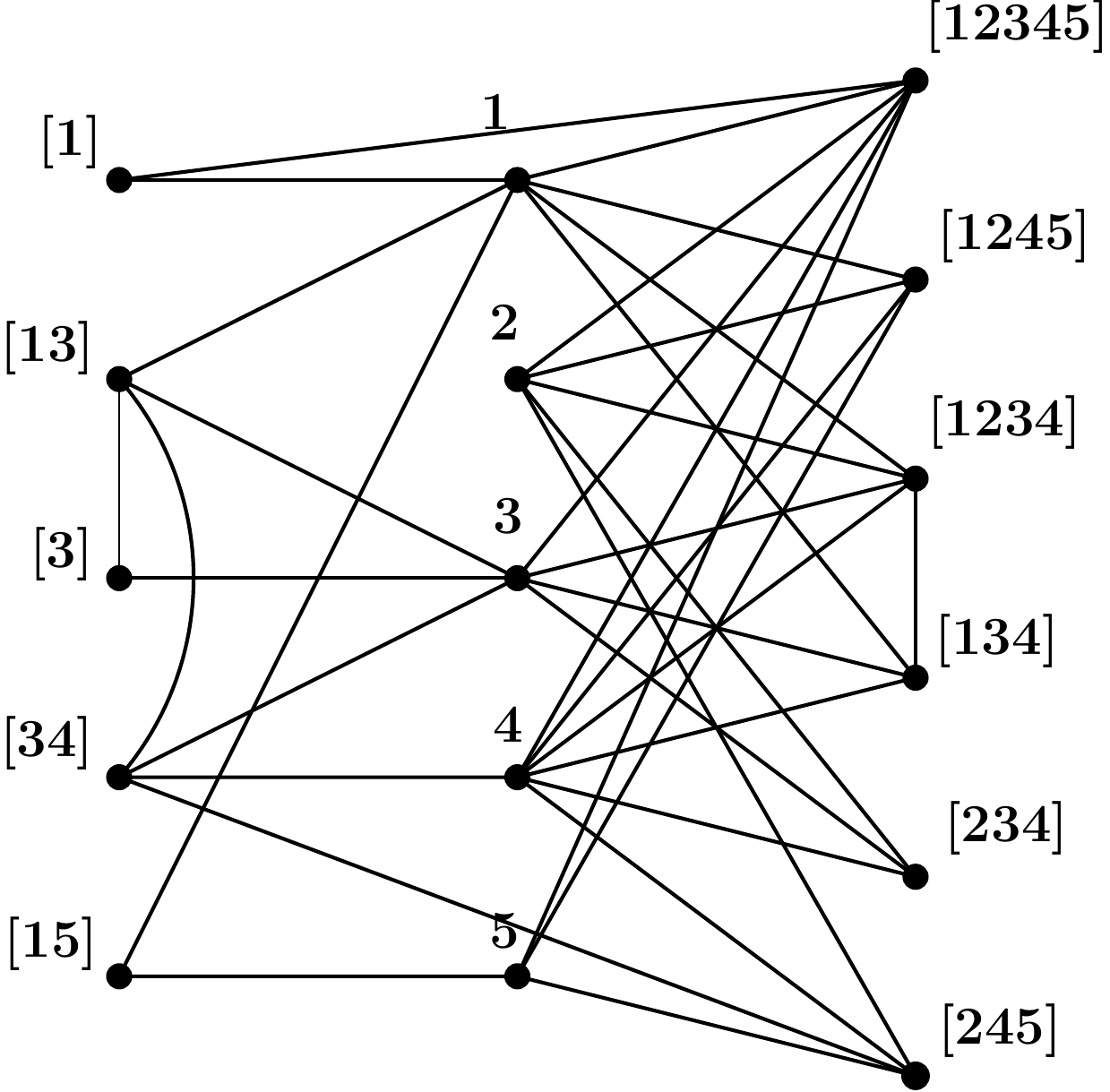}\hspace{1cm}\includegraphics[height=5cm]{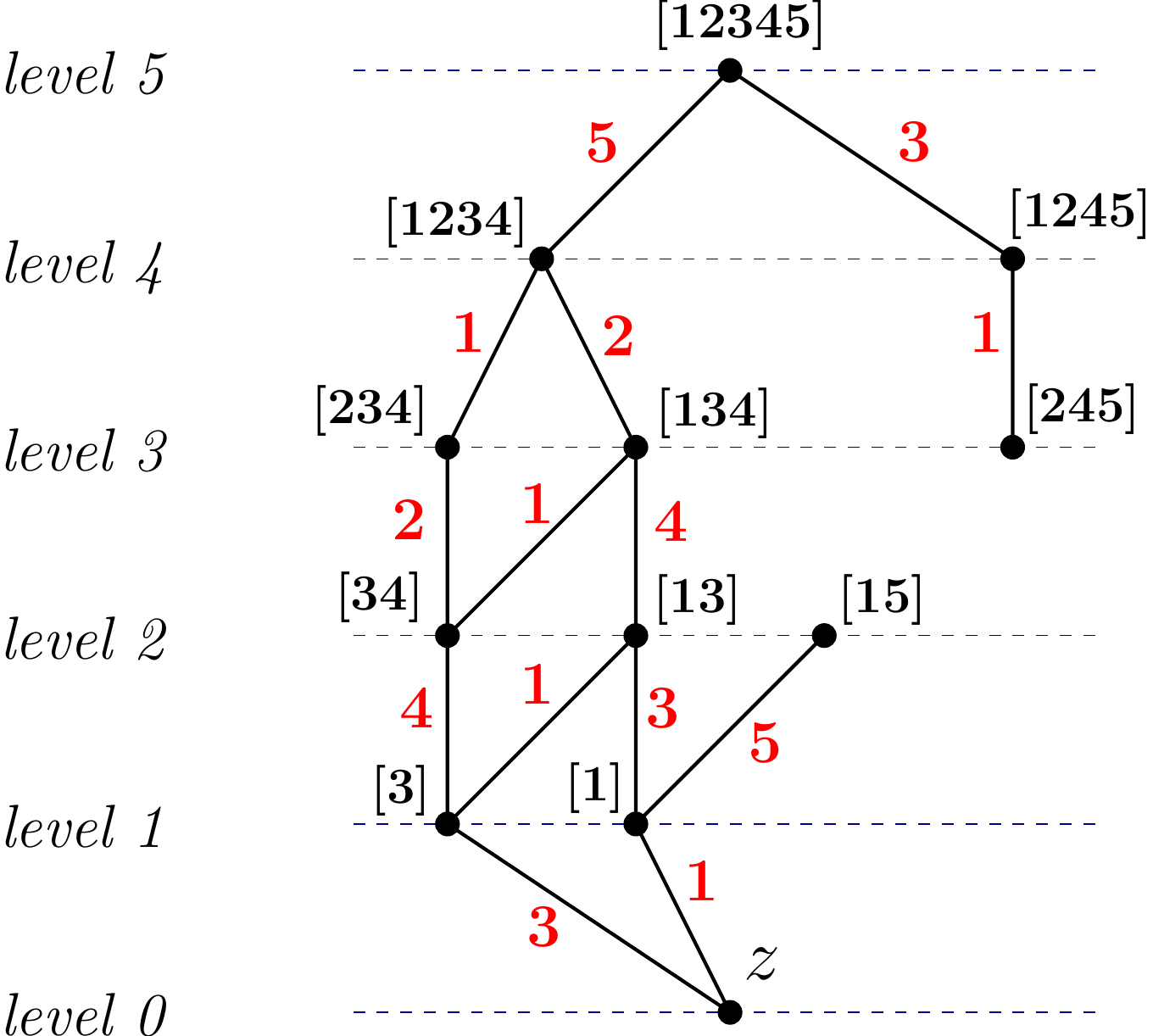}
\end{center}
\caption{Left: a graph $G$. Right: the LD-set-associated graph $G^{S}$, where $S =\{ 1,2,3,4,5 \}$ .}
\label{fig.exempleGestrella}
\end{figure}

Notice that two vertices of $V\setminus S$ are adjacent in $G^{S}$ if their neighborhood in $S$ differ in exactly one vertex, the label of the edge, and $z$ is adjacent to vertices of $V\setminus S$ with exactly a neighbor in $S$.
Therefore, we can represent the graph $G^{S}$ with the vertices lying on $|S |+1$ levels, from bottom (level $0$) to top (level $|S |$), in such a way that vertices with exactly $k$ neighbors in $S$ are at level $k$. There is at most one vertex at level $|S |$ and, if it is so, this vertex is adjacent to all vertices of $S$. The vertices at level 1 are those with exactly one neighbor in $S$ and $z$ is the unique vertex at level $0$. An edge of $G^{S}$ has its endpoints at consecutive levels. Moreover, if $e=xy\in E(G^{S})$, with $\ell (e)=u\in S$, and $x$ is at exactly one level higher than $y$, then $N(x)\cap S=(N(y)\cap S)\cup \{ u \}$, i.e., $x$ and $y$ have the same neighborhood in $S \setminus \{ u \}$.
Therefore, the existence of an edge in $G^{S }$ with label $u\in S$ means that $S\setminus \{ u \}$ is not an LD-set. Hence, if $S$ is an LD-code, then for every $u\in S$ there exists at least an edge in $G^{S }$ with label $u$.
See Figure \ref{fig.exempleGestrella} for an example of an LD-set-associated graph.

The following proposition states some properties of LD-set-associated graphs.

\begin{proposition}\label{claims123}
Let $S$ be  an LD-set with exactly $k$ vertices of a connected graph $G=(V,E)$ of order $n$.
Let $G^{S}$ be its  $S$-associated graph. Then the following holds.
\begin{enumerate}
\item $|V(G^{S})|=n-k+1$.

\item $G^{S}$ is bipartite.

\item Incident edges have different labels.

\item Every cycle of $G^{S}$ contains an even number of edges labeled $v$, for all $v\in S$.

\item Let $\rho$  be a walk with no repeated edges in $G^{S}$.
If  $\rho$ contains an even number of edges labeled $v$ for every $v\in S$, then $\rho$  is a closed walk.

\item If $\rho =x_ix_{i+1}\dots x_{i+h}$ is a path satisfying that vertex $x_{i+h}$ lies at level $i+h$, for any $h\in \{ 0,1,\dots ,h \}$, then
     \begin{enumerate}
       \item
       the edges of $\rho$ have different labels;
       %all the labels of the edges of $\rho$  are different;
       %$\ell (x_ix_{i+1})$,   $\ell (x_{i+1}x_{i+2})$, $\dots $, $\ell (x_{i+h-1}x_{i+2h})$
       \item for all  $j\in \{ i+1,i+2,\dots ,i+h \}$, $N(x_j)\cap S$ contains the vertex $\ell (x_kx_{k+1})$, for any $k\in \{ i,i+1,\dots , j-1\}$.
     \end{enumerate}

\end{enumerate}
\end{proposition}
\begin{proof}
\begin{enumerate}
  \item It is a direct consequence  from the definition of $G^{S}$.
	
  \item Consider the sets $V_1=\{ x\in V(G^{S}) : |N(x)\cap S |\textrm{ is odd}\}$ and  $V_2=\{ x\in V(G^{S}) : |N(x)\cap S|\textrm{ is even}\}$.
	Then $V(G^{S})=V_1\cup V_2$ and $V_1\cap V_2 =\emptyset$.
	Since $||N(x)\cap S|-|N(y)\cap S||=1$ for any $xy\in E(G^{S})$, it is clear that the vertices $x,y$ are not in the same subset $V_i$, $i=1,2$.
	
  \item Suppose that edges $e_1=xy$  and $e_2=yz$ have the same label $l(e_1)=l(e_2)=v$.
	This means that the sets $N(x)\cap S$ and $N(y)\cap S$ differ only in element $v$ and the sets $N(y)\cap S$ and $N(z)\cap S$ differ only in element $v\in S$.
	It is only possible if $N(x)\cap S=N(z)\cap S$, implying that $x=z$.
	
  \item Let $\rho$ be a cycle such that $E(\rho)=\{x_0x_1,x_1x_2,\ldots x_hx_0\}$.
	The set of neighbors in $S$ of two consecutive vertices differ exactly in one vertex.
	If we begin with $N(x_0)\cap S$, each time we add (remove) the vertex of the label of the corresponding edge, we have to remove (add) it later in order to obtain finally the same neighborhood,  $N(x_0)\cap S$.
	Therefore, $\rho$ contains an even number of edges with label $v$.

\item
Consider the vertices $x_0,x_1,x_2,x_3,...,x_{2k}$ of the walk $\rho$.
In this case, $N(x_{2k})\cap S$ is obtained from $N(x_0)\cap S$ by adding or removing  the labels of all the edges of the walk.
Since every label appears an even number of times, for each element $v\in S$ we can match its appearances in pairs, and each pair means that we  add and remove (or remove and add) it from the neighborhood in $S$.
Therefore, $N(x_{2k})\cap S=N(x_0)\cap S$, and hence $x_0=x_{2k}$.

\item It straightly follows  from the fact that $N(x_{j})\cap S = N(x_{j-1}\cap S)\cup \{ \ell (x_{j-1}x_{j})\}$, for any $j\in \{ i+1,\dots ,i+h\}$.
\end{enumerate}
\vspace{-.8cm}\end{proof}

%\newpage
%%%%%%%%%%%%%%%%%%%%%%%%%%%%%%%%%%%%%%%%%%%%%%%%%%%%%%%
%%%%%%%%%%%%%%%%%%%%%%%%%%%%%%%%%%%%%%%%%%%%%%%%%%%%%%%
%%%%%%%%%%%%%%%%%%%%%%%%%%%%%%%%%%%%%%%%%%%%%%%%%%%%%%%
%%%%%%%%%%%%%%%%%%%%%%%%%%%%%%%%%%%%%%%%%%%%%%%%%%%%%%%
%%%%%%%%%%%%%%%%%%%%%%%%%%%%%%%%%%%%%%%%%%%%%%%%%%%%%%%
\section{The bipartite case}

In the sequel,  $G=(V,E)$ stands for  a bipartite connected graph of order $n=r+s\ge4$, such that  $V=U\cup W$, being $U$,$W$ their stable sets and  $1\le |U|=r\le s=|W|$.

This section is devoted to solving the equation $\lambda (\overline{G}) =\lambda (G)  + 1$  when we restrict ourselves to bipartite graphs.
According to  Corollary \ref{cori}, this equality is feasible only for graphs without global LD-codes.

\begin{lemma}\label{prop.UW} Let $S$ be an LD-code of
%a bipartite graph
$G$. Then, $\lambda (\overline{G})\le \lambda (G)$ if any of the following conditions holds.
\begin{enumerate}
  \item $S\cap U\not= \emptyset$ and $S\cap W\not= \emptyset$.
  \item $r<s$ and  $S=W$.
	\item $2^r\le s$.
\end{enumerate}
\end{lemma}
\begin{proof}
	If $S$ satisfies item 1., then there is no vertex dominating $S$ and, by Proposition \ref{pro.vertexdom}, $S$ is a global LD-code of $G$, which, according to Corollary  \ref{cori}, means that $\lambda (\overline{G})\le \lambda (G)$.
	Next, assume that  $r<s$ and  $S=W$. 
	In this case, $U$ is not an LD-set, but is a dominating set since $G$ is connected. Therefore, there exists a pair of  vertices $w_1,w_2\in W$ such that $N(w_1)=N(w_2)$. Hence,  $W-\{ w_1 \}$ is an LD-set of $G-w_1$. Let $u\in U$ be a vertex adjacent to $w_1$ (it exists since $G$ is connected), and notice that
	$(W\setminus \{ w_1 \})\cup \{ u \}$ is an LD-code of $G$ with vertices in both stable sets, which,  by the preceding item, means that $\lambda (\overline{G})\le \lambda (G)$.
	Finally, if  $2^r\le s$ then  $S\neq U$, which means that $S$ satisfies either item 1. or item 2.
\end{proof}

\begin{corollary}\label{corin}
%\merce{Nou redactat per evitar problemes quan $r=s$.}
If $\lambda (\overline{G})=\lambda (G)+1$, then $r\le s \le 2^r-1$.
Moreover, if $r<s$ then $U$ is the  unique LD-code of $G$, and if $r=s$ we may assume that $U$ is a non-global LD-code of $G$.
\end{corollary}

%%%%%%%%%%%%%%%%%%%%%%%%%%%%%%%%%%%%%%%%%%%%%%%%%%%%%
\begin{proposition}\label{pro.r1o2}
%Let $G$ be a bipartite graph  of
If $G$ has order at least 3 and $1\le r \le 2$, then $\lambda (\overline{G}) \le \lambda (G)$.
\end{proposition}
\begin{proof}
If $r=1$, then $G$ is the star $K_{1,n-1}$ and $\lambda (\overline{G})=\lambda (G)=n-1$.

\begin{figure}[!hbt]
\begin{center}
\includegraphics[width=0.6\textwidth]{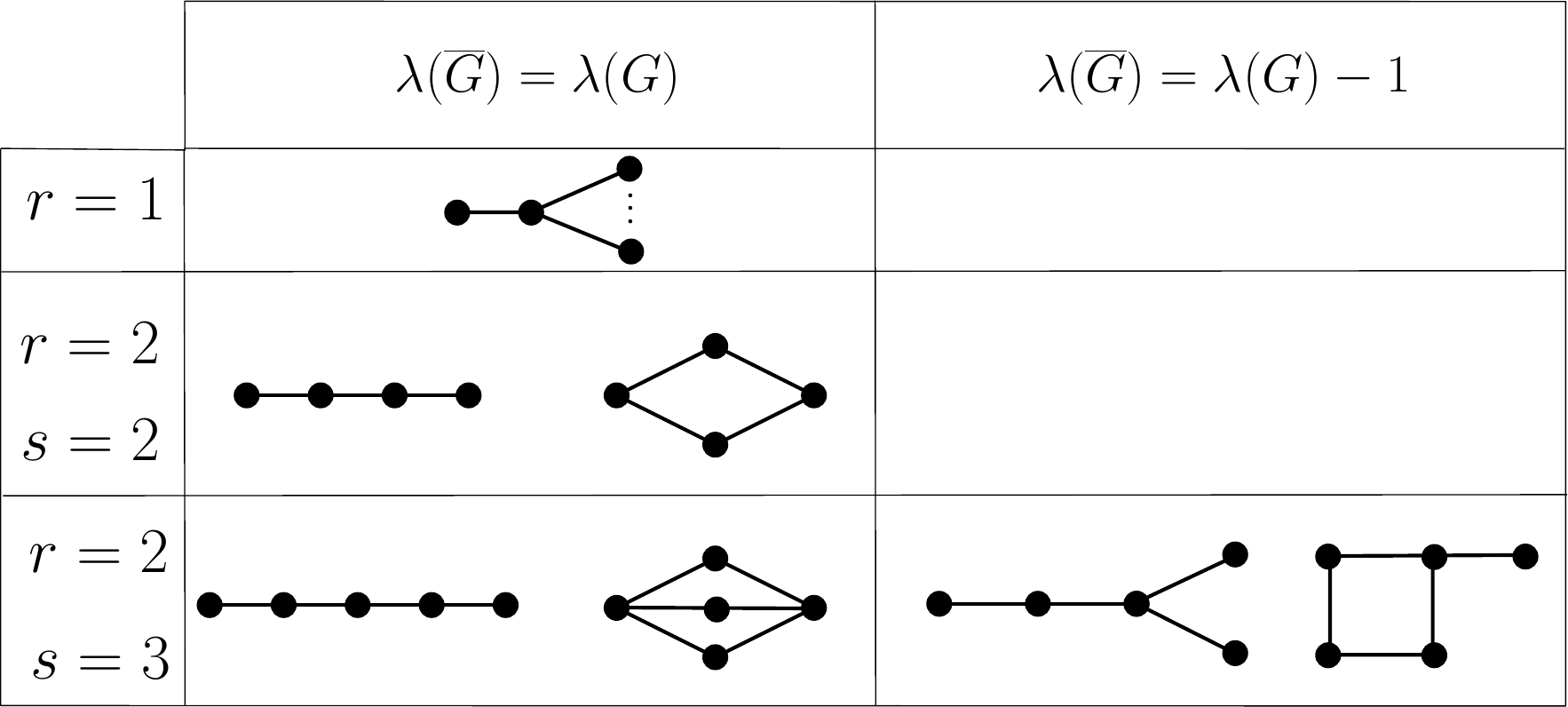}
\caption{Some bipartite graphs with $1\le r \le 2$.}\label{fig.r1o2}
\end{center}
\end{figure}

Suppose  that $r=2$.
If $s\ge 2^2=4$ then,  by Lemma \ref{prop.UW}, $\lambda (\overline{G})\le \lambda (G)$.

If $s=2$,  then $G$ is either $P_4$ and $\lambda (\overline{P_4})= \lambda (P_4)=2$, or $G$ is $C_4$ and $\lambda (\overline{C_4})= \lambda (C_4)=2$.

If $s=3$, then $G$ is $P_5$, $K_{2,3}$,  $K_2(1,2)$,  or a banner $P$, and
$\lambda (\overline{P_5})= \lambda (P_5)=2$,
$\lambda (\overline{K_{2,3}})= \lambda (K_{2,3})=3$,
$2=\lambda (\overline{K_2(1,2)})< \lambda (K_2(1,3))=3$, and
$2=\lambda (\overline{P})<\lambda (P)=3$.
\end{proof}
%%%%%%%%%%%%%%%%%%%%%%%%%%%%%%%%%%%%%%%%%%%%%%%%%%%%%

Notice that the only bipartite graphs $G$ such that $\lambda (G)=2$ are  $P_3$, $P_4$,  $C_4$ and $P_5$.
Observe also that every bipartite graph $G$ such that $\lambda (\overline{G}) =\lambda (G)  + 1$ satisfies $\lambda (G)\ge r$, being $r$ the order of its smallest stable set.

Next, we approach the case $\lambda (G)\ge3$.  That is to say, from now on we assume that $r\ge 3$.

%%%%%%%%%%%%%%%%%%%%%%%%%%%%%%
\begin{lemma}\label{lem.exist}
%Let $G=(U\cup W,E)$ a bipartite graph such that $r\le|U|\le|W|$ and
%\merce{Nou redactat}
If $\lambda (\overline{G})= \lambda (G)+1$ and $U$ is an LD-code of $G$, then $G^{U}$ contains, for every vertex $u\in U$, at least two edges with label $u$.
\end{lemma}
%%%%%%%%%%%%%%%%%%%%%%%%%%%%%%
\begin{proof}
Condition  $\lambda (\overline{G})= \lambda (G)+1$ implies  that there is no LD-code of $G$ with vertices in both stable sets.
Therefore, for any $u\in U$, $U\setminus \{ u \}$ is not an LD-set of the graph $G-u$, otherwise the set $U\setminus \{ u \}$ together with a neighbor of vertex $u$ would be an LD-code of $G$ with vertices in both stable sets. 
We distinguish two possible cases.

Case (a). If $N(U\setminus \{ u \})=W$ there is at least a pair of vertices $w_{1},w_{2} \in W$ such that $N(w_{1})\bigtriangleup N(w_{2})=\{ u \}$ (see Figure \ref{fig.condicioBip},(a)).
Moreover, since there is no LD-code with vertices in both stable sets, there must be another pair of vertices $w_{3},w_{4} \in W$ such that $N(w_{3})\bigtriangleup N(w_{4})=\{ u \}$, 
otherwise $(U\setminus \{ u \}  )\cup \{ w \} $, where $w$ is the neighbor of $u$ in $\{ w_{1},w_{2}\}$, would be an LD-code with vertices in both stable sets. 

Case (b). If $N(U\setminus \{ u \})\subsetneq W$, then there is exactly a vertex $w\in W$ such that $N(w)=\{ u\}$ (see Figure \ref{fig.condicioBip},(b)).
By the other hand, if the neighborhood in $U\setminus \{ u \}$ of any two vertices of $W\setminus \{ w \}$ is different, then $(U\setminus \{ u \}) \cup \{ w \}$ would be an LD-code with vertices in both stable sets. 
Therefore, there is at least a pair of vertices $w_{1},w_{2} \in W\setminus \{w\}$ such that $N(w_{1})\bigtriangleup N(w_{2})=\{ u \}$.  
Notice that in this case $N(w)\bigtriangleup \emptyset = \{ u \}$.

%%ESTE JODE
\begin{figure}[!hbt]
\begin{center}
\includegraphics[width=0.6\textwidth]{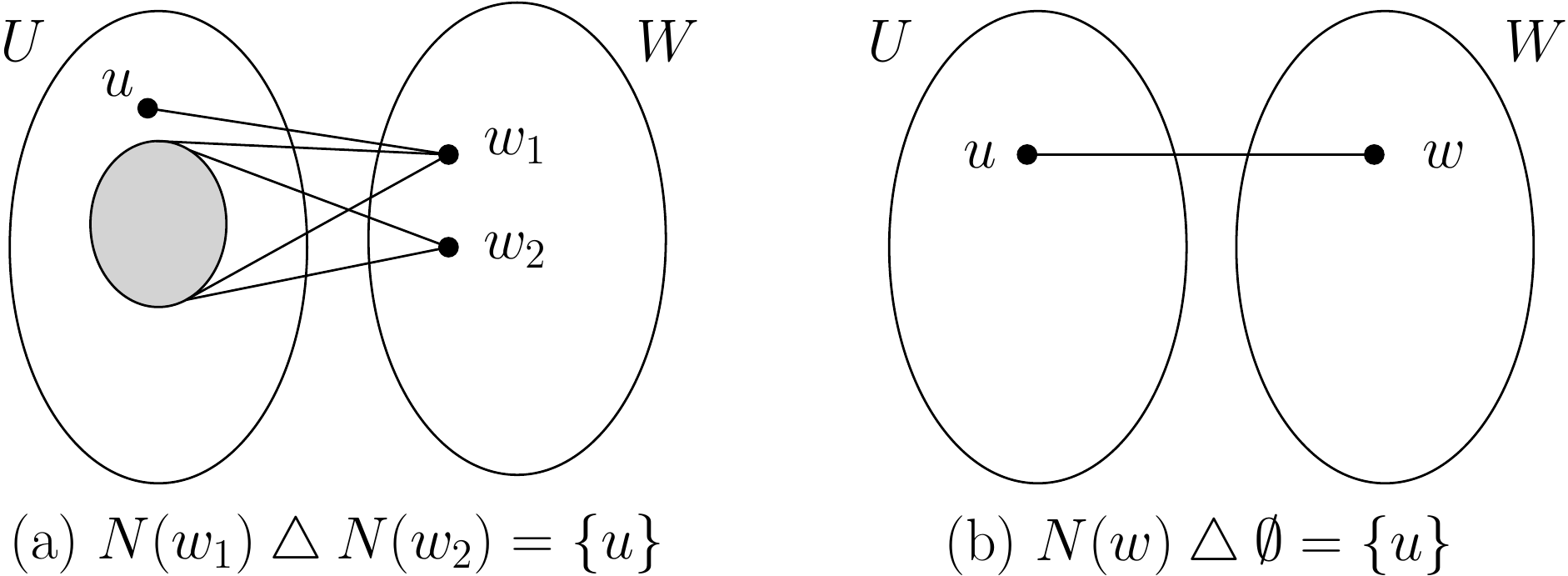}
\caption{Case (a): $N(U\setminus \{ u \})=W$. Case (b): $N(U\setminus \{ u \})\subsetneq W$.}\label{fig.condicioBip}
\end{center}
\end{figure}

%\begin{figure}[!hbt]
%\begin{center}
%\includegraphics[height=7cm]{bipartitsGuanyaCompl}
%\caption{All connected components of the subgraph $H$ are cactus.}\label{fig.acc}
%\end{center}
%\end{figure}

Consequently, in both cases, for every $u\in U$, there are at least two edges with label $u$ in the graph $G^U$.
\end{proof}

%%%  def de cactus aquí
In the study of LD-sets using the LD-associated graph, a family of graphs is particularly useful, the \emph{cactus graph} family.
A \emph{block} of a graph is a maximal connected subgraph with no cut vertices.
A connected graph $G$ is a \emph{cactus} if all its blocks are cycles or edges.
Cactus are characterized as those connected graphs with no edge shared by two cycles.

%\newpage
%%%%%%%%%%%%%%%%%%%%%%%%%%%%%%
\begin{lemma}\label{claim.subgraphcactus}
%Let $G=(U\cup W,E)$ be a bipartite graph such that $r\le|U|\le|W|$ and $\lambda (\overline{G})= \lambda (G)+1$.
Let $\lambda (\overline{G})= \lambda (G)+1$
and assume that $U$ is an LD-code of $G$. Consider a subgraph $H$ of $G^U$ induced by a set of edges containing exactly two edges with label $u$, for each $u\in U$.
Then, all  connected components of $H$ are cactus.
\end{lemma}
%%%%%%%%%%%%%%%%%%%%%%%%%%%%%%
\begin{proof}
We will prove that there is no edge lying on two different cycles of $H$.
Suppose on the contrary that there is an edge $e_1$ contained in two different cycles $C_1$ and $C_2$ of $H$.
If the label of $e_1$ is $u\in U$, by Proposition \ref{claims123} both cycles $C_1$ and $C_2$ contain the other edge $e_2$ of $H$ labeled with $u$.
Suppose that $e_1=x_1y_1$ and $e_2=x_2y_2$ and assume w.l.o.g. that there exist $x_1-x_2$ and $y_1-y_2$ paths in $C_1$ not containing edges $e_1,e_2$. Let $P_1$ and $P_1'$ denote respectively those paths (see Figure \ref{fig.ac} a).

We have two possibilities for $C_2$:
(i) there are  $x_1-x_2$ and $y_1-y_2$ paths in $C_2$ not containing neither $e_1$ nor $e_2$. Let $P_2$ denote the $x_1-x_2$ path in $C_2$ in that case (see Figure \ref{fig.ac} b);
(ii) there are  $x_1-y_2$ and $y_1-x_2$ paths in $C_2$ not containing neither $e_1$ nor $e_2$ (see Figure \ref{fig.ac} c).

In  case (ii), the closed walk formed with the  path $P_1$, $e_1$ and the $y_1-x_2$ path in $C_2$ would contain a cycle with exactly an edge labeled with $u$, which is a contradiction (see Figure \ref{fig.ac} d).

In case (i), at least one the following cases holds: the $x_1-x_2$ paths in $C_1$ and in $C_2$, $P_1$ and $P_2$, are different or the $y_1-y_2$ paths in $C_1$ and in $C_2$ are different (otherwise, $C_1=C_2$).

\begin{figure}[!hbt]
\begin{center}
\includegraphics[height=7cm]{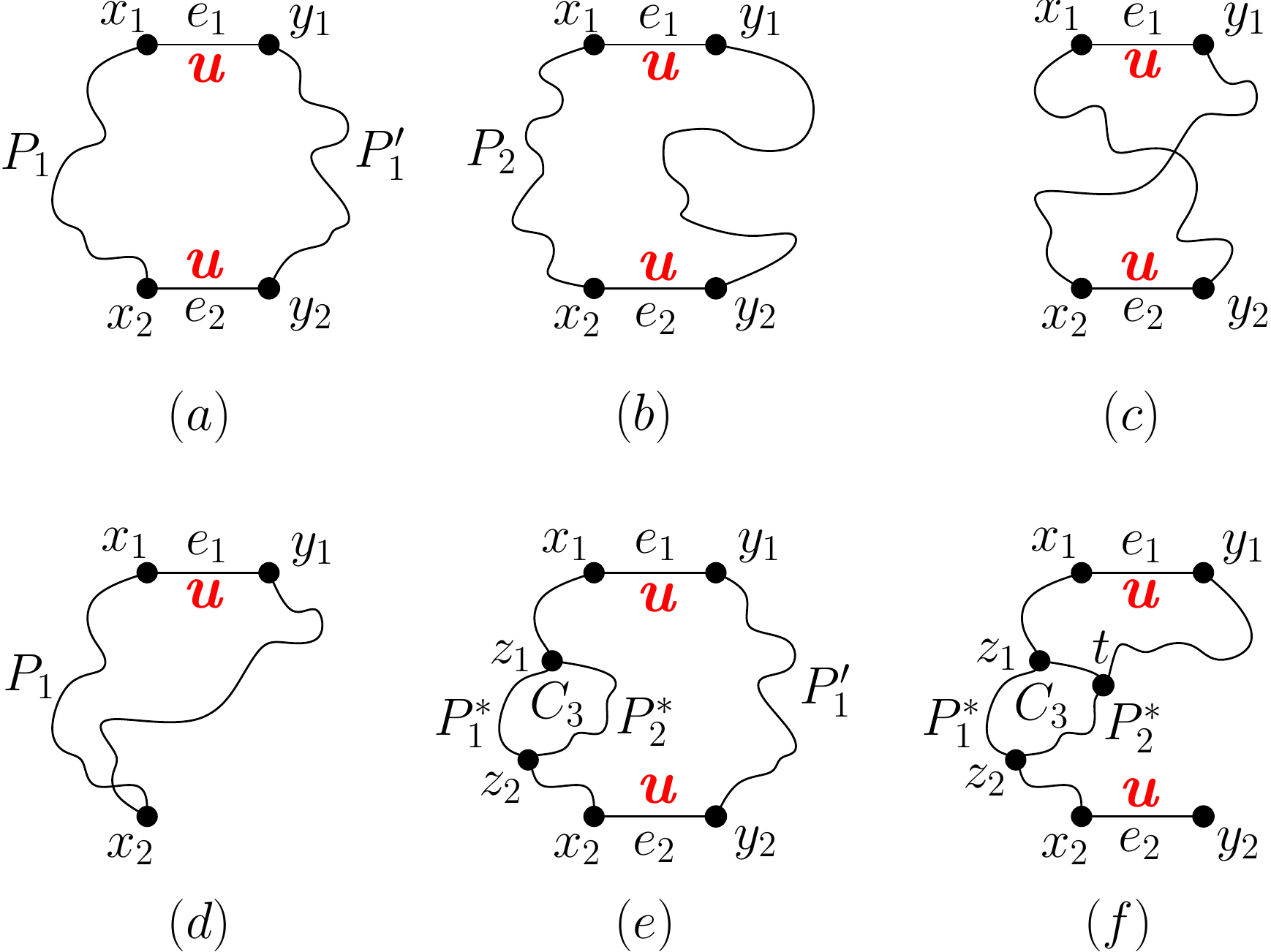}
\caption{All connected components of the subgraph $H$ are cactus.}\label{fig.ac}
\end{center}
\end{figure}

%\merce{Nova demostració}
Assume that $P_1$ and  $P_2$ are different.
Let $z_1$ be the last vertex shared by $P_1$ and $P_2$ advancing from $x_1$ and let $z_2$ be the first vertex shared by $P_1$ and $P_2$ advancing from $z_1$ in $P_2$.
Notice that $z_1\not= z_2$.
Consider the cycle $C_3$ formed with the $z_1-z_2$ paths in $P_1$ and $P_2$.
Let $P_1^*$ and $P_2^*$ be respectively the $z_1-z_2$  subpaths of $P_1$ and $P_2$ (see Figure \ref{fig.ac} e).
We claim that the internal vertices of $P_2^*$ do not lie in $P_1'$.
Otherwise, consider the first vertex $t$ of $P_1'$ lying also in $P_2^*$.
The cycle beginning in $x_1$, formed by the edge $e_1$, the $y_1-t$  path contained in $P_1'$, the $t-z_1$ path contained in $P_2^*$ and the $z_1-x_1$ path contained in $P_1$ has exactly one appearance of an edge with label $u$, which is a contradiction
(see Figure \ref{fig.ac} f).
By Proposition \ref{claims123}, the labels of edges belonging to  $P_1^*$  appear exactly two times in cycle $C_3$, but they also appear exactly two times in cycle $C_1$.
But this is only possible if they appear exactly two times in $P_1^*$, since $H$ contains exactly to edges with the same label.
By Proposition \ref{claims123}, $P_1^*$ must be a  closed path, which is a contradiction.
\end{proof}

We present next some properties relating parameters of bipartite graphs having cactus as connected components.

\begin{lemma}\label{lemacactus}
Let $H$ be a bipartite graph of order at least 4 such that all its connected components are cactus.
If $H$  has $\cc(H)$ connected components and $\cy(H)$ cycles, then the following holds.

\begin{enumerate}
\item
$|V(H)|=|E(H)|-\cy(H)+\cc (H)$.

\item
If $\ex (H)=|E(H)|-4\, \cy(H)$,  then $\ex (H)\ge 0$ and $|V(H)|=\frac 34 |E(H)|+\frac 14 \ex (H)+\cc (H)$.

\item
$|V(H)|\ge \frac 34 |E(H)|+1$.

\item
$|V(H)|=\frac 34 |E(H)|+1$ if and only if $H$ is connected and all blocks are cycles of order 4.

\end{enumerate}
\end{lemma}
\begin{proof}
\begin{enumerate}

\item Since $H$ is a planar graph with $\cy(H)+1$ faces and $\cc(H)$ connected components, the equality follows from the generalization of Euler's Formula:
$$(\cy(H)+1)+|V(H)|=|E(H)|+(\cc(H)+1).$$

\item All cycles of a bipartite graph have at least 4 edges, hence $\ex (H)\ge 0$. By the preceding item,
$$|V(H)|=|E(H)|-\cy(H)+\cc (H)=|E(H)|-\frac 14 (|E(H)-\ex (H))+\cc (H)=
\frac 34 |E(H)|+\frac 14 \ex (H) +\cc (H).$$

\item It  immediately follows   from the preceding item.

\item  Observe first that if $H$ is connected and all blocks are cycles of order $4$, then $\cc (H)=1$ and $|E(H)|=4\, \cy(H)$. Hence, $\ex (H)=|E(H)|-4\, \cy(H)=0$ and by item 2,
$|V(H)|= \frac 34 |E(H)|+1.$
%consequently $|V(H)|= \frac 34 |E(H)|+\frac 14 \ex (H) +\cc (H)=\frac 34 |E(H)|+1.$

Conversely, suppose that $|V(H)|= \frac 34 |E(H)|+1.$
The graph  $H$ must be connected, since otherwise
  $|V(H)|= \frac 34 |E(H)|+\frac 14 \ex (H) +\cc (H)\ge \frac 34 |E(H)| +2$.
	On the other hand, if $H$ contains a cycle of order at least $6$ or a bridge, then $\ex (H)=|E(H)|-4\, \cy(H)>0$, implying that
  $|V(H)|= \frac 34 |E(H)|+\frac 14 \ex (H) +\cc (H)> \frac 34 |E(H)| +\cc (H)=\frac 34 |E(H)| +1$.
\end{enumerate}
\vspace{-.8cm}\end{proof}

%\newpage
%%%%%%%%%%%%%%%%%%%%%%%%%%%%%%
\begin{proposition}\label{prop.complementariguanya}
%Let $G$ be a bipartite graph  such that $ r \ge 3$.
 If $ r \ge 3$ and $\lambda (\overline{G})=\lambda (G)+1$, then $\frac{3r}2\le s\le 2^r-1$.
\end{proposition}
%%%%%%%%%%%%%%%%%%%%%%%%%%%%%%
\begin{proof}
By Corollary \ref{corin}, we have that $s\le 2^r-1$, and we may assume that $U$ is a non-global LD-code and there is no LD-code with vertices in both stable sets.

Consider a subgraph $H$ of $G^{U}$ with exactly two edges with label $u$ for any $u\in U$.
The graph $H$  is bipartite since it is a subgraph of $G^{U}$ and by Lemma  \ref{lemacactus},
%$$2r=|E(H^*)|\le \frac 43 (|V(H^*)|-1)\le  \frac 43 (|V(G^*)|-1)= \frac 43 s,$$
$$s+1=|V(G^U)| \ge |V(H)| \ge \frac 34 \, |E(H)| + 1 =\frac 34 \, (2r) + 1= \frac {3r}2 +1 $$
 and consequently $s\ge \frac {3r}2$.
%On the other hand, if $s> 2^r-1$, then it is not possible that $U$ be an LD-set of $G$.
%Therefore, if $S$ is an LD-code of $G$ then $S$ has vertices at both stable sets or $S=W$.
%By Proposition \ref{prop.UW}, in both cases $\lambda (\overline{G})\le \lambda (G)$.
\end{proof}

%%%%%%%%%%%%%%%%%%%%%%%%%%%%%%%%%%%%%%%%%%%%%%%%%%%%%%%%%%%%%
\begin{lemma}\label{ultimlema}
%\merce{Nou redactat}
%Let $G=(U\cup W,E)$ a bipartite graph such that $r\le|U|\le|W|=s$ and $\lambda (\overline{G})= \lambda (G)+1$.
If $\lambda (\overline{G})= \lambda (G)+1$ and $U$ is an LD-code of $G$, let $z$ be the vertex of $G^{U}$ introduced in Definition \ref{def.Gomega} and let $H$ be a subgraph of $G^{U}$ with exactly two edges with label $u$, for each $u\in U$.
Then the following holds.
\begin{enumerate}

\item If $H$ has at least two connected components, then $s\ge \frac {3r} 2+1$.

\item If $\deg_{G^U} (z)=0$, then $s\ge \frac {3r} 2+1$.

\item  $\deg_{G^U} (z)\not= 0$ if and only if there is at least a vertex in $W$ of degree 1 in $G$.

\item If $G$ has no vertex of degree 1 in $W$, then  $s\ge \frac {3r} 2+1$.

\end{enumerate}

\end{lemma}
\begin{proof}
\begin{enumerate}

\item By Lemma \ref{lemacactus},
$s+1\ge |V(H)| =\frac 34 |E(H)|+\frac 14 \ex (H) +\cc (H)\ge \frac 34 |E(H)|+2=\frac {3r}2+2$,
and thus, $s\ge \frac {3r}2+1$.

\item If $\deg_{G^U} (z)=0$, then $z$ is not a vertex of $H$.
Hence,
$s\ge |V(H)| =\frac 34 |E(H)|+\frac 14 \ex (H) +\cc (H)\ge \frac 34 |E(H)|+1=\frac {3r}2+1$.

\item We know that $\deg_{G^U} (z)\not= 0$ if and only if there is a vertex $w\in W$ satisfying  $N(w)\bigtriangleup N(z)=N(w)\bigtriangleup \emptyset = \{ u\}$, i.e. if and only if $\deg_G(w)=1$.

\item It is a straight consequence of items 2 and 3.

\end{enumerate}
\vspace{-1cm}\end{proof}

%\newpage
%%%%%%%%%%%%%%%%%%%%%%%%%%%%%%%%%%%%%%%%%%%%%%%%%%%%%%%%%%%%%
\begin{proposition}\label{tontito}
There are no bipartite graphs $G$ satisfying $\lambda (\overline{G})=\lambda (G)+1$ if $\frac{3r}2\le s<\frac{3r}2+1$.
\end{proposition}
\begin{proof}
Suppose on the contrary that $G$ is a bipartite graph satisfying the conditions of the proposition.
Condition $\lambda (\overline{G})=\lambda (G)+1$ implies that we may assume that $U$ is an LD-code of $G$, there is no LD-code with vertices in both stable sets and $U$ is not an LD-set of $\overline{G}$.
Consider a subgraph $H$ of $G^U$ with exactly two edges with label $u$, for each $u\in U$ (it exists by Lemma \ref{lem.exist}).

Observe that the inequality is only possible for
$s=\frac{3r}2$, whenever $r$ is even, and for $s=\frac{3r+1}2$, whenever $r$ is odd.
If $r$ is even and $s=\frac{3r}2$, then
$$\frac {3r}2+1=s+1=|V(G^U)|\ge |V(H)|=\frac 34 |E(H)|+\frac 14 \ex (H)+\cc (H)=\frac {3r}2 +\frac 14 \ex (H)+\cc (H).$$
Since $\ex (H)\ge 0$ and $\cc (H)\ge 1$, this is only possible for $\ex (H)=0$, $\cc (H)=1$, and $V(G^U)=V(H)$.
By Lemma \ref{lemacactus}, $H$ is a cactus with all blocks cycles of order 4, concretely, $\frac r2$ cycles.
If $r$ is odd and $s=\frac{3r+1}2$, then
$$\frac{3r}2+\frac 24 +1=\frac {3r+1}2+1=s+1=|V(G^U)|\ge |V(H)|=\frac 34 |E(H)|+\frac 14 \ex (H)+\cc (H)=\frac {3r}2 +\frac 14 \ex (H)+\cc (H).$$
This is only possible for $\ex (H)=2$, $\cc (H)=1$, and  $V(G^U)=V(H)$.
By Lemma \ref{lemacactus}, $H$ is a cactus with exactly $\frac {r-1}2$ cycles: $\frac {r-1}2-1$ cycles of order 4 and a cycle of order 6, or $\frac {r-1}2$ cycles of order 4 and two bridges.

%\merce{Canvi de redacció}
We also know that  condition $\lambda (\overline{G})=\lambda (G)$ implies the existence of a vertex $w^*\in V(G)\subseteq V(G^U)=V(H)$ such that $N_G(w^*)=U$, i.e., $H$ has a vertex at the highest level.
Lemma \ref{ultimlema} allows us to conclude  that $H$ is connected and $z\in V(H)$.
Thus, $H$ must be a chain of cycles of order 4, or a chain of a cycle of order 6 and cycles of order 4, or a chain of a bridge and cycles of order 4, plus another bridge hanging from a vertex of this chain, with both bridges having the same label and, by Proposition \ref{claims123}, not lying in a path with all vertices at different levels (see Figure \ref{fig.chain}).
%\merce{Fi del canvi de redacció}

\begin{figure}[!hbt]
\begin{center}
\includegraphics[width=0.25\textwidth]{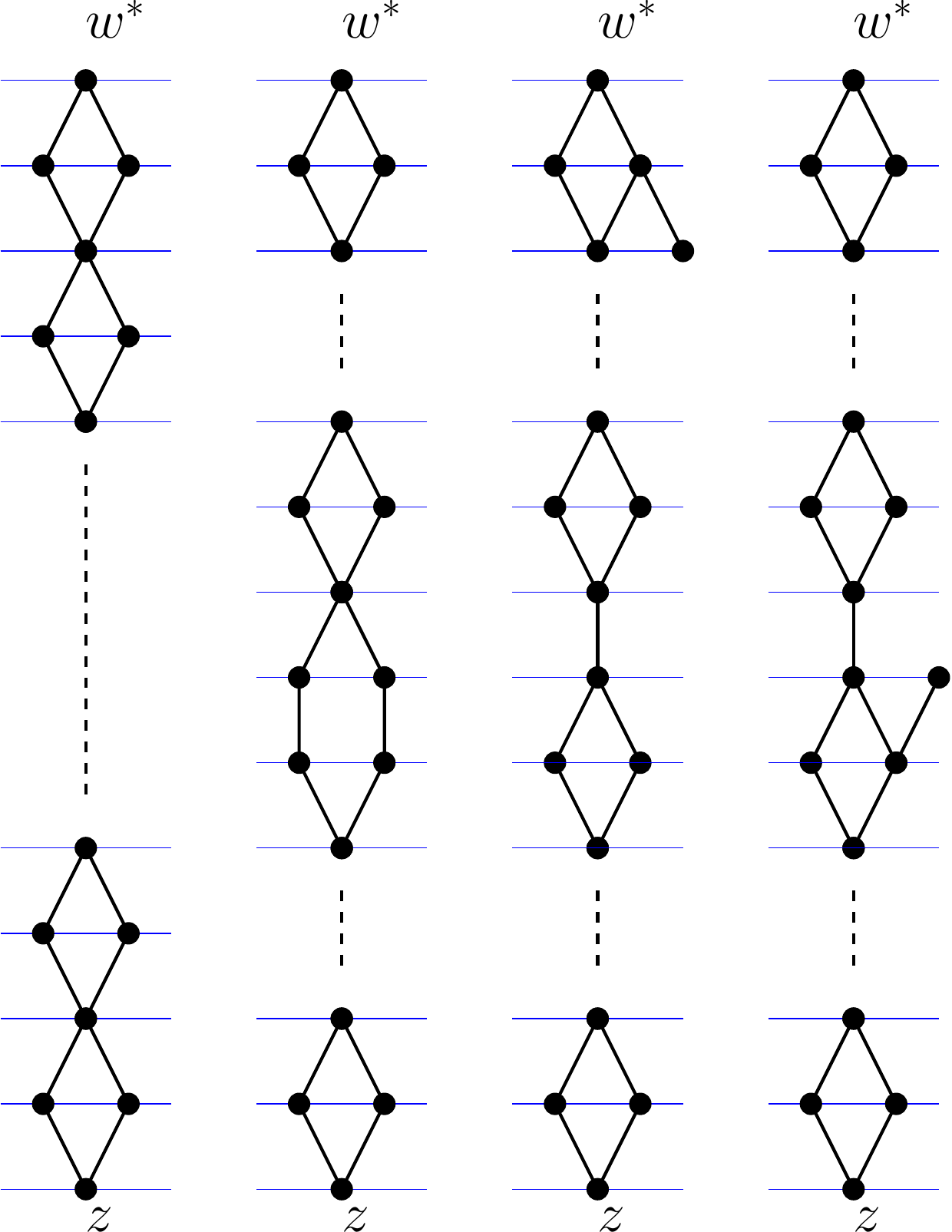}
\caption{Examples of subgraphs $H$.}\label{fig.chain}
\end{center}
\end{figure}

In consequence, one of the following cases holds in $H$:
(i) $z$ belongs to a cycle $C$ of order $4$;
(ii) $z$ belongs to a cycle $C$ of order $6$;
(iii) $z$ belongs to a bridge, $e$. In this case, there is no $x-z$ path of length $i$ in $H$ with consecutive vertices in levels $i,i-1,\dots ,1,0$ respectively containing both edges of $H$ with label $\ell (e)$.
We may assume w.l.o.g. that the labels $a,b,c\in U$ of the edges of $C$ and $e$ are those of Figure \ref{fig.noExisteix}.
Let $w_0$ be the vertex of $G$ indicated in the same figure.

\begin{figure}[!hbt]
\begin{center}
\includegraphics[width=0.4\textwidth]{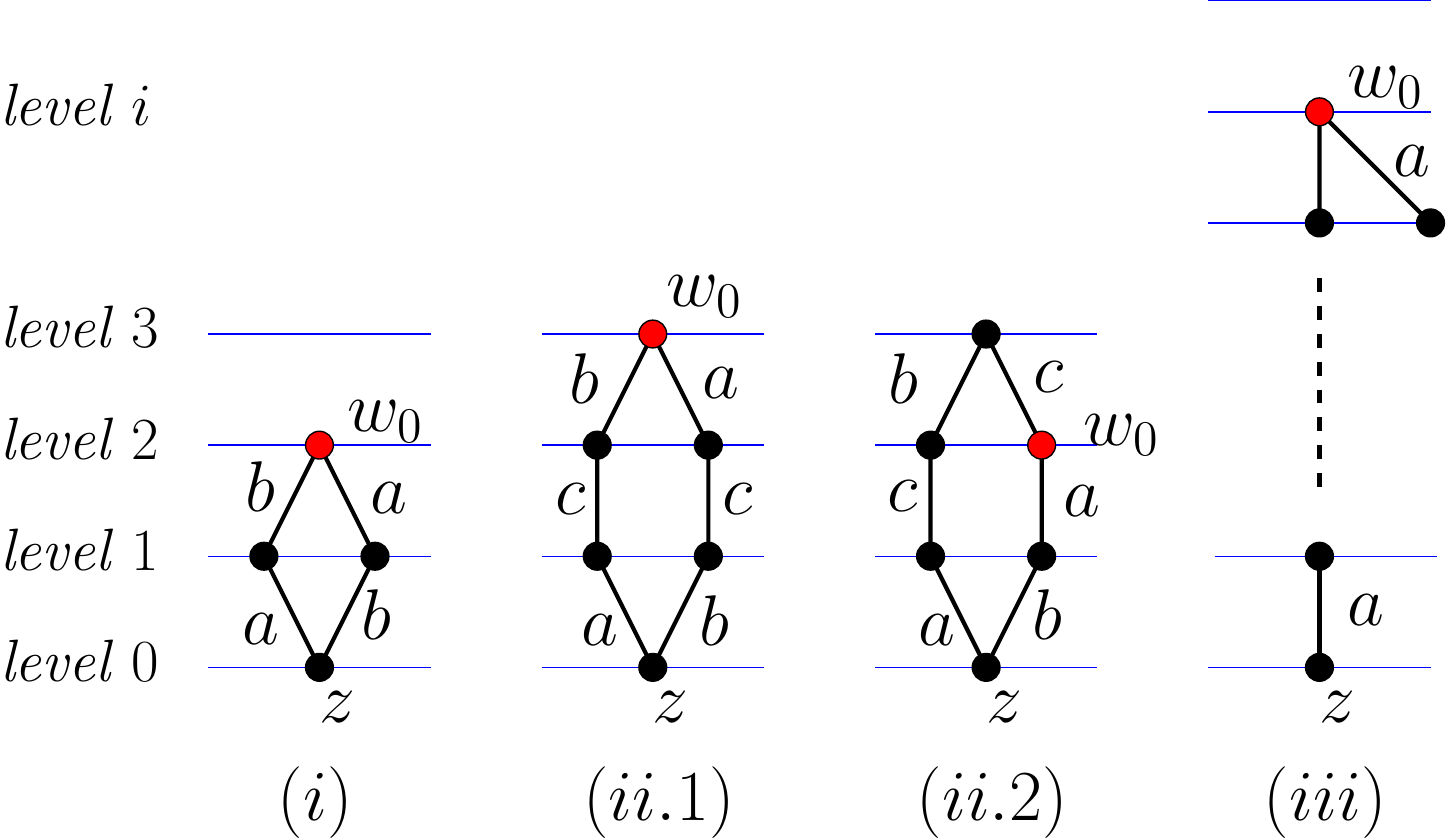}
\caption{Possible cases for vertex $z$ in subgraph $H$.}\label{fig.noExisteix}
\end{center}
\end{figure}

We claim that the set $S=(U\setminus \{ a \})\cup \{ w_0 \}$ is an LD-set of $\overline{G}$ with exactly $r$ vertices.
Indeed, if $w_0\not= w^*$, then $N_{\overline{G}}(a)\cap S=S\setminus \{ w_0 \}$, $N_{\overline{G}}(w^*)\cap S=\{w_0\}$ and for any $x\in W\setminus \{ w^*,w_0\}$, $N_{\overline{G}}(x)\cap S=\{w_0\}\cup S'$, where $S'=U\setminus (N_G(x)\cup \{ a\})\not= \emptyset$, since $N_G(x)\not= U\setminus \{ a \}$.
Moreover, for any pair of different vertices $x,y\in W\setminus \{ w^*,w_0\}$, $N_G(x)\cap (U\setminus \{ a \})\not= N_G(y)\cap (U\setminus \{ a \})$, implies that $N_{\overline{G}}(x)\cap S\not= N_{\overline{G}}(y)\cap S$.
If $w_0=w^*$, then $N_{\overline{G}}(a)\cap S=S\setminus \{ w^* \}$,
and for any $x\in W\setminus \{ w^*\}$, $N_{\overline{G}}(x)\cap S=\{w^*\}\cup S'$, where $S'=U\setminus (N_G(x)\cup \{ a\})$.
Moreover, for any pair of different vertices $x,y\in W\setminus \{ w^*\}$, $N_G(x)\cap (U\setminus \{ a \})\not= N_G(y)\cap (U\setminus \{ a \})$, implies that $N_{\overline{G}}(x)\cap S\not= N_{\overline{G}}(y)\cap S$.
\end{proof}

%\newpage
%%%%%%%%%%%%%%%%%%%%%%%%%%%%%%%%%%%%%%%%%%%%%%%%%%%%%%%%%%%%%
\begin{proposition}\label{prop.construccio}
  For every pair $(r,s)$, $r,s\in \mathbb{N}$, such that $3\le r$ and $\frac{3r}2 +1\le s\le 2^r-1$, there exists a bipartite graph $G(r,s)$ such that $\lambda (\overline{G})=\lambda (G)+1$.
\end{proposition}
  \begin{proof}
    Let $s=\Big\lceil \frac{3r}2 +1 \Big\rceil$.
		Consider the bipartite graph $G(r,\Big\lceil \frac{3r}2 +1 \Big\rceil)$ such that $V=U\cup W$, $U=[r]=\{ 1,2,\dots ,r\}$, and
    $W\subseteq \mathcal{P}([r])\setminus \{ \emptyset \}$ is defined as follows.
		For $r=2k$ even:
\begin{align*}
      W = & \Big\{ [r] \Big\}
           \cup \Big\{ [r]\setminus \{ i \} : i\in [r] \Big\}
           \cup \Big\{ [r]\setminus \{ 2i-1,2i\} : 1\le i\le k \Big\}\\
    \end{align*}
and for $r=2k+1$  odd:
\begin{align*}
      W = & \Big\{ [r] \Big\}
     \cup \Big\{ [r]\setminus \{ i \} : i\in [r] \Big\}
     \cup \Big\{ [r]\setminus \{ 2i-1,2i\} : 1\le i\le k-1 \Big\}\\
          & \cup \Big\{ [r]\setminus \{ r-2,r-1 \}, [r]\setminus \{ r-1,,r \}, [r]\setminus \{ r-2,r-1,r \} \Big\}
    \end{align*}

\begin{figure}[!hbt]
\begin{center}
\includegraphics[height=3.0cm]{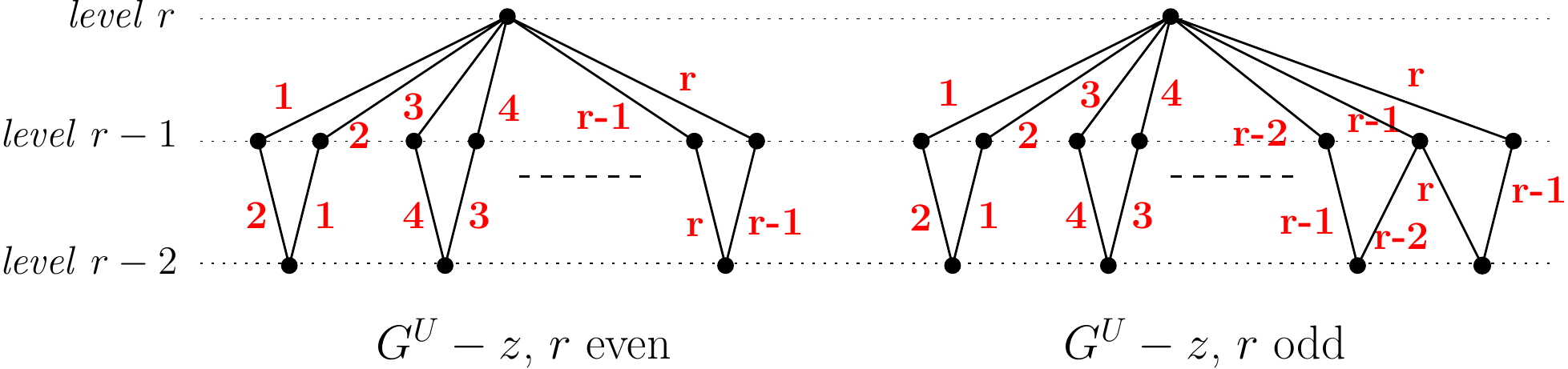}
\caption{The labeled graph $G^U-z$, for $G=G(r,\Big\lceil \frac{3r}2 +1 \Big\rceil)$ and $U=\{1, \dots , r \}$.}\label{fig.exemples}
\end{center}
\end{figure}

By construction, $U$ is an  LD-set of $G$ with $r$ vertices and by Corollary \ref{teoremon}, $U$ is not an LD-set of $\overline{G}$ (see in Figure \ref{fig.exemples} the $U$-associated graph, $G^U$).
We claim that there is no LD-set in $\overline{G}$ with at most $r$ vertices.

Suppose that $S$ is an LD-set of $\overline{G}$.
We already know that $S\not= U$.
Let us assume that $|S\cap U|=r-k$, $k\ge 1$.
Consider the subgraph $H$ of $G^U$ induced by $2k$ edges of $G^U$ with label $u\in U\setminus S$.
Notice that, by definition, this subgraph exists and $z\notin V(H)$.
Moreover, by Lemma \ref{claim.subgraphcactus}, all  connected components of $H$ are cactus.
Observe that, by definition of the associated graph $G^U$, the vertices lying at the same connected component of $H$ have the same neighborhood in $S\cap U$.
We know also that $W$ induces a complete graph in $\overline{G}$.
Therefore, at least all but one vertex of each connected component of $H$ must be in $S$.
By Lemma \ref{lemacactus}, this value is $$|V(H)|-\cc (H)=\frac 34 |E(H)|+\frac 14 \ex (H)=\frac 34 2k +\frac 14 \ex (H)=\frac 32 k + \frac 14 \ex (H)\ge \frac 32 k.$$
Hence, $|S|\ge (r-k)+ \frac 32 k=r+\frac 12 k>r$.

\emph{Remark.} We derive from this result that $\lambda (G)=r$.
Nevertheless, a direct proof of this fact can be given: it can be proved in a similar way that there is no LD-set of $G$ with less than $r$ vertices.

For $s>\lceil \frac{3r}2 +1 \rceil$, we can add up to $2^{r}-1-r$ vertices to the set $W$ of the  graph $G(r,\Big\lceil \frac{3r}2 +1 \Big\rceil)$ taking into account that the neighborhoods in $U$  of the vertices of $W$ must be different and non-empty.
\end{proof}
%%%%%%%%%%%%%%%%%%%%%%%%%%%%%%%%%%%%%%%%%%%%%%%%%%%%%%%%%%%%%%%%%%%%%%

%%%%%%%%%%%%%%%%%%%%%%%%%%%%%%%%%%%%%%%%%%%%%%%%%%%%%%%%%%%%%%%%%%%%%%
\begin{theorem}
  Let $r,s$ be a pair of integers such that $3\le r\le s$. 
  \begin{enumerate}
	
\item[(1)] There exists a bipartite graph $V(G)=U\cup W$ such that $|U|=r$, $|W|=s$ and $\lambda (\overline{G})=\lambda (G)-1$.

\item[(2)] There exists a bipartite graph $V(G)=U\cup W$ such that $|U|=r$, $|W|=s$ and $\lambda (\overline{G})=\lambda (G)$.
		
\item[(3)] There exist a bipartite graph $V(G)=U\cup W$ such that $|U|=r$, $|W|=s$ and  $\lambda (\overline{G})=\lambda (G)+1$ if and only if $\frac {3r}2+1 \le s \le 2^r-1$.

\end{enumerate}
	
\end{theorem}
\begin{proof}
To prove item (1), take the bi-star $K_2({r,s})$ and check that $\lambda (K_2({r,s}))=r+s-2$ and $\lambda (\overline{K_2({r,s})})=r+s-3$.
To prove item (2), take the biclique $K_{r,s}$ and check that $\lambda (K_{r,s})=\lambda (\overline{K_{r,s}})=r+s-2$.
Finally, observe that item (3) is a corollary of Propositions \ref{prop.complementariguanya}, \ref{tontito} and  Proposition \ref{prop.construccio}.
\end{proof}
%%%%%%%%%%%%%%%%%%%%%%%%%%%%%%%%%%%%%%%%%%%%%%%%%%%%%%%%%%%%%%%%%%%%%%

%\newpage

%%%%%%%%%%%%%%%%%%%%%%%%%%%%%%%%%%%%%%%%%%%%%%%%%%%%%%%%
%%%%%%%%%%%%%%%%%%%%%%%%%%%%%%%%%%%%%%%%%%%%%%%%%%%%%%%%
%\section*{Acknowledgements}
%Research partially supported by projects
%MTM2012-30951,
%Gen. Cat. DGR 2009SGR1040,
%%Gen. Cat. DGR 2014SGR46,
%ESF EUROCORES programme EUROGIGA-ComPoSe IP04-MICINN,
%MTM2011-28800-C02-01,
%Gen. Cat. DGR 2009SGR1387

%IÑAKI    MTM2011-28800-C02-01   2009 SGR 1387

%%%%%%%%%%%%%%%%%%%%%%%%%%%%%%%%%%%%%%%%%%%%%%%%%%%%%%%
%%%%%%%%%%%%%%%%%%%%%%%%%%%%%%%%%%%%%%%%%%%%%%%%%%%%%%%

%%%%%%%%%%%%%%%%%%%%%%%%%%%%%%%%%%%%%%%%%%%%%%%%%%%%%%%
%%%%%%%%%%%%%%%%%%%%%%%%%%%%%%%%%%%%%%%%%%%%%%%%%%%%%%%
%%%%%%%%%%%%%%%%%%%%%%%%%%%%%%%%%%%%%%%%%%%%%%%%%%%%%%%
%%%%%%%%%%%%%%%%%%%%%%%%%%%%%%%%%%%%%%%%%%%%%%%%%%%%%%%


\begin{thebibliography}{99}


%%%%%%%%%%%%
\bibitem{bchl} N. Bertrand, I. Charon, O. Hudry, A. Lobstein,
Identifying and locating-dominating codes on chains and cycles,
{\it Eur. J.  Combin.}, {\bf 25} (2004) 969--987.




%%%%%%%%%%%%
\bibitem{bcmms07} M. Blidia, M. Chellali, F. Maffray, J. Moncel, A. Semri,
Locating-domination and identifying codes in trees,
{\it Australas. J. Combin.}, {\bf 39} (2007) 219--232.


%%%%%%%%%%%%
\bibitem{brca98} R.C. Brigham, J.R. Carrington,
Global domination, in: T.W. Haynes, S.T. Hedetniemi, P.J. Slater (Eds.),
Domination in Graphs, Advanced Topics, Marcel Dekker, New York, 1998, pp. 30--320.



%%%%%%%%%%%%%%%%%%
\bibitem{brdu90}
R. C. Brigham, R. D. Dutton,
Factor domination in graphs,
{\it Discrete Math.},{\bf 86} (1-3) (1990) 127--136.


%%%%%%%%%%%%%%%%%%
\bibitem{cahemopepu12}
J. C\'{a}ceres, C. Hernando, M. Mora, I. M. Pelayo, M. L. Puertas,
Locating-dominating codes: Bounds and extremal cardinalities,
{\it Appl. Math. Comput.},{\bf 220} (2013) 38--45.



%%%%%%%%%%%%
\bibitem{chlezh11} G. Chartrand, L. Lesniak, P. Zhang,
Graphs and Digraphs, fifth edition,
{CRC Press}, Boca Raton (FL), (2011).





%%%%%%%%%%%%%%%%%%
\bibitem{clm11}
C. Chen, R. C. Lu, Z. Miao,
Identifying codes and locating-dominating sets on paths and cycles,
{\it Discrete Appl. Math.},{\bf 159} (15) (2011) 1540--1547.


%%%%%%%%%%%
\bibitem{hahesl} T. W. Haynes, S. T. Hedetniemi and P. J. Slater,
\emph{Fundamentals of domination in graphs},
Marcel Dekker, New York, 1998.


%%%%%%%%%%%%
\bibitem{ours3} C. Hernando, M. Mora, I. M. Pelayo,
Nordhaus-Gaddum bounds for locating domination,
{\it Eur. J.  Combin.}, {\bf 36} (2014) 1--6.


%%%%%%%%%%%%
\bibitem{oursglobal} C. Hernando, M. Mora, I. M. Pelayo,
On global location-domination in graphs, http://arxiv.org/abs/1312.0772, 2014.



%%%%%%%%%%%%
\bibitem{hola06} I. Honkala, T. Laihonen,
On locating-dominating sets in infinite grids,
{\it Eur. J.  Combin.}, {\bf 27} (2) (2006) 218--227.


%%%%%%%%%%%%%%%%%%
\bibitem{lobstein}
A. Lobstein, Watching systems, identifying, locating-dominating ans discriminating codes in graphs,
\newline http://www.infres.enst.fr/~lobstein/debutBIBidetlocdom.pdf



%%%%%%%%%%%%%%%%%%
\bibitem{rasl84}
D. F. Rall, P. J. Slater, On location-domination numbers for certain classes of graphs,
{\it Congr. Numer.} {\bf 45} ( (1984) 97--106.



%%%%%%%%%%%%%%%%%%
\bibitem{sam89}
E. Sampathkumar, The global domination number of a graph,
{\it J. Math. Phys. Sci.} {\bf 23} ( (1989) 377--385.



%%%%%%%%%%%%
\bibitem{slater88} P. J. Slater,
Dominating and reference sets in a graph,
{\it J. Math. Phys. Sci.} {\bf 22} (1988) 445--455.







\end{thebibliography}
\end{document}